\pdfoutput=1
\documentclass{amsart}
\usepackage{amsmath}
\usepackage{amsthm}
\usepackage{amsfonts}
\usepackage{amssymb}
\usepackage{amscd}
\usepackage{mathtools}
\usepackage{multirow}

\usepackage{graphicx}
\usepackage{color} 
\usepackage{transparent}

\theoremstyle{plain}
\newtheorem{theorem}{Theorem}[section]

\newtheorem{lemma}[theorem]{Lemma}
\newtheorem{proposition}[theorem]{Proposition}
\newtheorem{corollary}[theorem]{Corollary}
\theoremstyle{definition}
\newtheorem{definition}[theorem]{Definition}

\makeatletter
\newcommand{\newreptheorem}[2]{\newtheorem*{rep@#1}{\rep@title}\newenvironment{‌​rep#1}[1]{\def\rep@title{#2 \ref*{##1}}\begin{rep@#1}}{\end{rep@#1}}}
\makeatother

\numberwithin{equation}{section}
\numberwithin{figure}{section}

\newcommand{\R}{\ensuremath \mathbb{R}}

\newcommand{\Z}{\ensuremath \mathbb{Z}}
\newcommand{\N}{\ensuremath \mathbb{N}}

\newcommand{\inv}{^{-1}}

\newcommand{\vol[1]}{\ensuremath \text{vol($ #1 $)}}

\begin{document}

\title{The Spectra of Volume and Determinant Densities of Links}
\author{Stephan D. Burton}
\thanks{Supported by NSF Grants DMS-1105843 and DMS-1404754.}

\begin{abstract}
The \textit{volume density} of a hyperbolic link $K$ is defined to be the ratio of the hyperbolic volume of $K$ to the crossing number of $K$. We show that there are sequences of non-alternating links with volume density approaching $v_8$, where $v_8$ is the volume of the ideal hyperbolic octahedron. We show that the set of volume densities is dense in $[0,v_8]$. The \textit{determinant density} of a link $K$ is $[2 \pi \log \det(K)]/c(K)$. We prove that the closure of the set of determinant densities contains the set $[0, v_8]$.
\end{abstract}

\maketitle

\section{Introduction}
Recent work in knot theory has studied the relationship between the hyperbolic volume of a knot or link and invariants of a knot or link that can be read from a diagram. See for example \cite{Guts}, \cite{FKPtwists}, \cite{FKP}, \cite{BanachSurvey}, \cite{LackenbyBound}, and references therein. Recent work of Champanerkar, Kofman, and Purcell (\cite{qdensity}, \cite{CKP}, and \cite{CKP2}) investigated a relationship between the volume of a knot, its determinant, and its crossing number. In this paper we continue to explore this relationship. We begin with the following invariant.

\begin{definition}
Given a hyperbolic link $K$, let $\text{vol}(K)$ be its hyperbolic volume, and let $c(K)$ be the crossing number of $K$. The \textit{volume density} of $K$ is defined to be $$d_{\text{vol}}(K): = \frac{\text{vol}(K)}{c(K)}$$
\end{definition}

D. Thurston \cite{DThurston} showed that $d_{\text{vol}}(K) \leq v_8$ for all knots $K$, where $v_8 \approx 3.66286$ is the volume of the regular ideal hyperbolic octahedron. This was done by decomposing $S^3 \backslash K$ into octahedra, placing one octahedron at each crossing, and pulling the remaining vertices to $\pm \infty$. Adams \cite{AdamsBound} showed that $\text{vol}(K) \leq (c(K)-5)v_8 + 4v_3$ for any link having $c(K) \geq 5$, where $v_3 \approx 1.01494$ is the volume of the ideal hyperbolic tetrahedron. Therefore $v_8$ is a strict upper bound for the volume density of any finite link.

Champanerkar, Kofman, and Purcell \cite{CKP} showed that the upper bound of $v_8$ on volume density is asymptotically sharp. More precisely, there exist sequences of links $\{K_n\}_{n = 0}^\infty$ such that $$\lim_{n \to \infty} d_{\text{vol}}(K_n) = v_8$$ (see Theorem \ref{thm:CKPthm}). Such sequences of links are called \textit{geometrically maximal}. The authors of \cite{CKP} produce examples of geometrically maximal sequences of links by constructing links approaching the \textit{infinite weave} $\mathcal{W}$, the infinite alternating link with the square lattice projection depicted in Figure \ref{fig:InfiniteWeave}. All the links in the examples constructed were alternating, and do not contain a cycle of tangles (see Definition \ref{def:Cycle}). 
The first purpose of this paper is to show that there exist geometrically maximal sequences of links having a cycle of tangles (see  Theorem \ref{thm:CycleEasy}). We then use this to prove Theorem \ref{thm:nonalt} which states that there exist geometrically maximal sequences of non-alternating links.

\begin{figure}
\def\svgwidth{250pt}
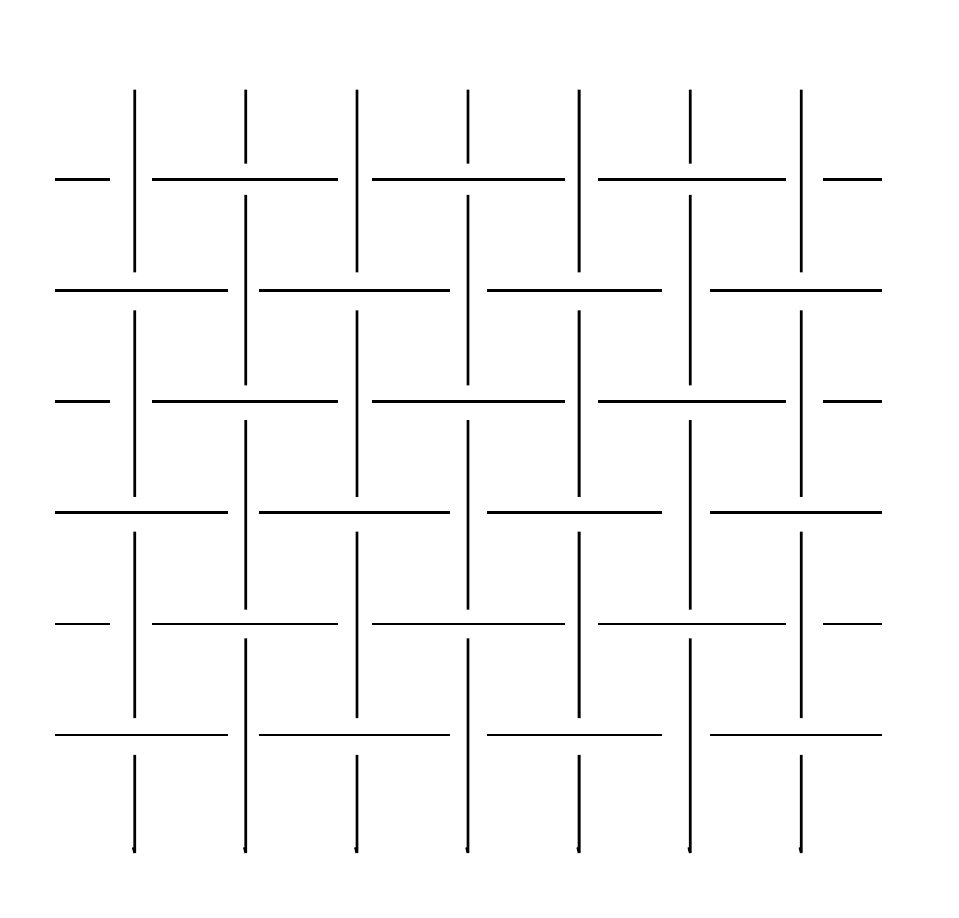
\caption{The infinite weave $\mathcal{W}$.}
\label{fig:InfiniteWeave}
\end{figure}

Let $\mathcal{C}_{\text{vol}} = \{\text{vol}(K)/c(K): K \text{ is a hyperbolic link}\} \subseteq \R$ and let $\text{Spec}_{\text{vol}}$ be the set of limit points of $\mathcal{C}_{\text{vol}}$. We call $\text{Spec}_{\text{vol}}$ the \textit{spectrum of volume densities}. A question that arose in \cite{CKP} was how one may describe the sets $\mathcal{C}_{\text{vol}}$ and $\text{Spec}_{\text{vol}}$. (See also \cite{qdensity} for discussion and related questions.) In particular, what numbers occur as volume densities? Since $v_8$ forms an upper bound on volume density, it is clear that $\mathcal{C}_{\text{vol}}$ and $\text{Spec}_{\text{vol}}$ are subsets of $[0,v_8]$. The existence of geometrically maximal knots implies that $v_8 \in \text{Spec}_{\text{vol}}$. In Lemma \ref{lem:MinimalKnot}, we describe \textit{geometrically minimal} sequences of links, i.e. those with volume density approaching 0. Hence $0 \in \text{Spec}_{\text{vol}}$. It was shown by Champanerkar, Kofman, and Purcell in \cite{CKP2} that $2v_3 \in \text{Spec}_{\text{vol}}$. We are able to prove the following theorem, which will be restated in Theorem \ref{thm:spectrum}.

\begin{theorem}\label{thm:IntroSpectrumVol}
The set $\mathcal{C}_{\emph{vol}}$ of volume densities of hyperbolic links is a dense subset of $[0,v_8]$, and $\emph{Spec}_{\emph{vol}} = [0,v_8]$. In other words, given $x \in [0,v_8]$ there exists a sequence $\{K_n\}_{n = 1}^\infty$ of hyperbolic links such that the volume densities satisfy $\displaystyle\lim_{n \to \infty} d_{\emph{vol}}(K_n) = x$.
\end{theorem}

Another invariant that was studied in \cite{qdensity} and \cite{CKP} was the \textit{determinant density}, $d_{\det}(K)$ of a link $K$, which is defined by the equation $$d_{\text{det}}(K):= \dfrac{2\, \pi \,\log(\det(K))}{c(K)}$$
It is shown in \cite{CKP} that there are sequences $\{K_n\}_{n = 1}^\infty$ of links such that $\lim_{n \to \infty} d_{\det}(K_n) = v_8$ (see Theorem \ref{thm:CKPthm2} of this paper). We define $\mathcal{C}_{\det} = \{d_{\det(K)}: K \text{ is a (not necessarily hyperbolic) link}\} \subseteq \R$ and define $\text{Spec}_{\text{vol}}$ to be the set of limit points of $\mathcal{C}_{\det}$. We call $\text{Spec}_{\det}$ the \textit{spectrum of determinant densities}. 

It is interesting to study the relationships between $\text{Spec}_{\text{vol}}$ and $\text{Spec}_{\det}$. While we know that $\text{Spec}_{\text{vol}} \subseteq [0,v_8]$, it is unknown whether $v_8$ forms an upper bound for $\text{Spec}_{\det}$. However, it is a conjecture \cite[Conjecture 1.1]{CKP} that $v_8$ is an upper bound for $\text{Spec}_{\det}$. We are able to prove the following theorem which will be restated in Theorem \ref{thm:spectrum2}.

\begin{theorem}\label{thm:IntroSpectrumDet}
The spectrum of determinant densities $\emph{Spec}_{\det}$ contains $[0,v_8]$. In other words, given any $x \in [0,v_8]$ there exists a sequence $\{K_n\}_{n = 1}^\infty$ of links satisfying $\displaystyle \lim_{n \to \infty} d_{\det}(K_n) = x$. 
\end{theorem}

In \cite{qdensity} and \cite{CKP} it is shown that $0, v_8 \in \text{Spec}_{\text{vol}} \cap \text{Spec}_{\det}$, and the authors ask what this intersection is. As a corollary of Theorems \ref{thm:IntroSpectrumVol} and \ref{thm:IntroSpectrumDet} we obtain:

\begin{corollary}
The intersection $\emph{Spec}_{\emph{vol}} \cap \emph{Spec}_{\det}$ is equal to $[0, v_8]$.
\end{corollary}

\noindent \textbf{Acknowledgements.} The author would like to thank his adviser Efstratia Kalfagianni for help in preparing the paper and for getting him started on this project. The author would also like to thank Jessica Purcell for conversations that were helpful, especially with the proof Theorem \ref{thm:IntroSpectrumDet}. Conversations with Ilya Kofman and Abhijit Champanerkar were also beneficial to the author.

\section{Background}

\subsection{Geometrically Maximal and Diagrammatically Maximal Links}
In this subsection we recall definitions and results from \cite{CKP}. Given a link $K \subset \R^3 \cup \{\infty\}$ we define the projection graph $G(K)$ to be the projection of $K$ onto the subset $(\R^2 \times \{0\}) \cup \{\infty\}$. This resulsts in a 4-valent graph in $\R^2 \cup \{\infty\} = S^2$.

\begin{definition}
Let $G$ be a possibly infinite graph. For any finite subgraph $H$, let $\partial H$ be the set of vertices of $H$ that share an edge with a vertex not in $H$. Let $| \cdot |$ denote the number of vertices in a finite graph. An exhaustive nested sequence of finite, connected subgraphs $\{H_n \subseteq G: H_n \subseteq H_{n+1}, \bigcup_{n = 0}^\infty H_n = G\}$ is a \textit{F\o lner sequence} for $G$ if 
$$ \lim_{n \to \infty} \dfrac{|\partial H_n|}{|H_n|} = 0.$$
The graph $G$ is \textit{amenable} if a F\o lner sequence for G exists. In particular, the infinite square lattice (i.e. the projection graph of $\mathcal{W}$) is amenable.
\end{definition}

Given a link $K$ in $S^3$, a \textit{Conway sphere} is a 2-sphere $S^2$ in $S^3$ intersecting $K$ transversely in four points. Using Conway spheres, we will define a cycle of tangles, which is an important definition for the statement of Theorem \ref{thm:CKPthm}.

\begin{definition}\label{def:Cycle}
A Conway sphere is called \textit{visible} if it is parallel to one dividing the diagram into two tangles, as in Figure \ref{fig:VisibleConwaySphere}. A tangle is called \textit{knotty} if it is nontrivial, and not a (portion of a) single twist region; i.e. not a rational tangle of type $n$ or $1/n$ for $n \in \Z$. We will say that $K_n$ contains a \textit{cycle of tangles} if $K_n$ contains a visible Conway sphere with a knotty tangle on each side.
\end{definition}

\begin{figure}
\def\svgwidth{250pt}
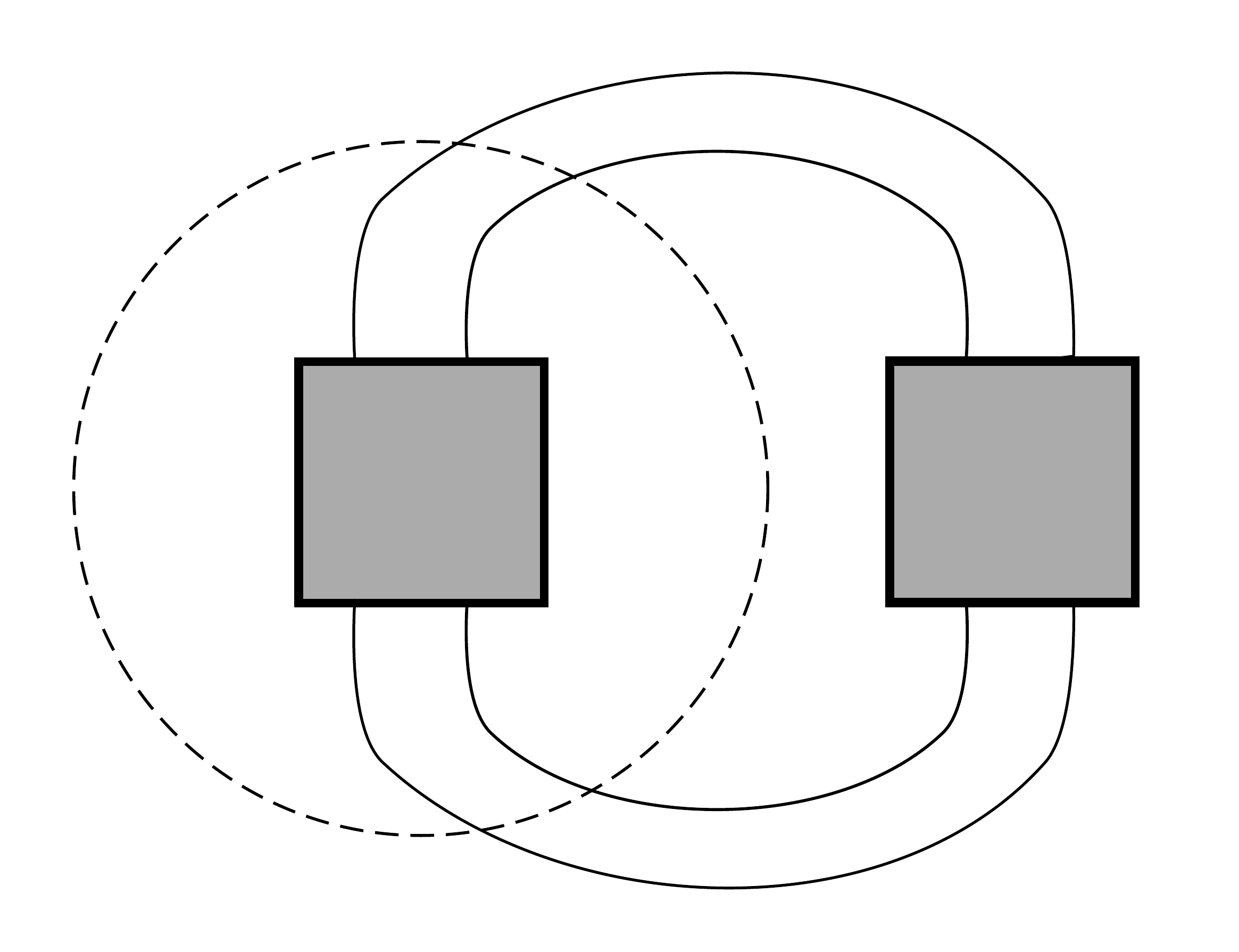
\caption{The dashed line represents a visible Conway sphere.}
\label{fig:VisibleConwaySphere}
\end{figure}

The following theorem was proved by Champanerkar, Kofman, and Purcell \cite[Theorem 1.4]{CKP}.

\begin{theorem}\label{thm:CKPthm}
Let $\{K_n\}_{n = 0}^\infty$ be a sequence of links with prime, alternating, twist-reduced diagrams that contain no cycle of tangles, such that 
\begin{enumerate}
\item there are subgraphs $G_n \subseteq G(K_n)$ that form a F\o lner sequence for $G(\mathcal{W})$, and
\item $\displaystyle \lim_{n \to \infty} \dfrac{|G_n|}{c(K_n)} = 1$.
\end{enumerate}
Then $\{K_n\}_{n = 0}^\infty$ is geometrically maximal.
\end{theorem}

The following theorem, proved in \cite[Theorem 1.5]{CKP}, gives a similar statement for diagrammatically maximal links.

\begin{theorem}\label{thm:CKPthm2}
Let $\{K_n\}_{n = 0}^\infty$ be a sequence of alternating link diagrams such that 
\begin{enumerate}
\item there are subgraphs $G_n \subseteq G(K_n)$ that form a F\o lner sequence for $G(\mathcal{W})$, and
\item $\displaystyle \lim_{n \to \infty} \dfrac{|G_n|}{c(K_n)} = 1$.
\end{enumerate}
Then $\{K_n\}_{n = 0}^\infty$ is diagrammatically maximal.
\end{theorem}

Note that the cycle of tangles condition is not necessary in the construction of \textit{diagrammatically} maximal links. This is one reason motivating Theorem \ref{thm:CycleEasy}, which states that there are \textit{geometrically} maximal sequences of links containing a cycle of tangles.

\subsection{Conway Sums and Belted Sums of Links}
Let $T$ be a tangle. There are two ways to close the tangle $T$ to form a link. Denote these closures $N(T)$ and $D(T)$ as depicted in Figure \ref{fig:TangleClosures}. Next, one may add an extra component, $C$, called the \textit{belt component}, to the link $N(T)$ in the following manner. Let $C$ be an unknotted circle that lies in a plane orthogonal to the projection plane and encircles the arcs added to $T$ to form $N(T)$, as in Figure \ref{fig:TangleClosures}. The resulting link, $B(T)$, is the \textit{belted link corresponding to $T$}.

\begin{figure}
\def\svgwidth{250pt}
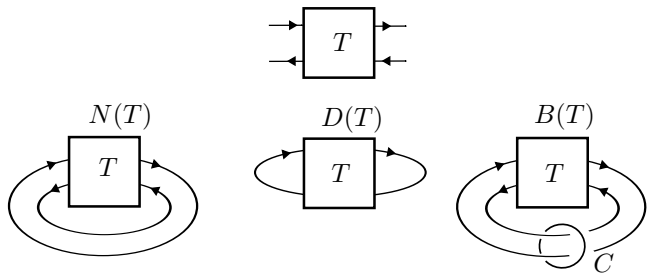
\caption{A tangle $T$, its two closures $N(T)$ and $D(T)$, and its corresponding belted tangle $B(T)$.}
\label{fig:TangleClosures}
\end{figure}

Given two tangles $T_1$ and $T_2$ such that $B(T_1)$ and $B(T_2)$ are hyperbolic, one may form the \textit{belted sum} of $B(T_1)$ and $B(T_2)$ as follows. Let $M_i$ be the link complement $S^3 \backslash B(T_i)$, for $i = 1, 2$. Work of Adams \cite{Adams} shows that the belt component $C_i$ of $T_i$ bounds a totally geodesic, twice-punctured disk in $M_i$. Cut each $M_i$ along this twice-punctured disk to form a manifold $M_i'$. Since there is a unique hyperbolic structure on the twice-punctured disk, we may glue $M_1'$ to $M_2'$ via an isometry of twice-punctured disks that maps $C_1$ to $C_2$. The result is a link, denoted $B(T_1) + B(T_2)$, which we call the belted sum of $B(T_1)$ and $B(T_2)$. See Figure \ref{fig:BeltedSum}. 

The following theorem about belted sums of links was proved by Adams \cite[Corollary 5.2]{Adams}. It will be important in the proofs of Theorems \ref{thm:IntroSpectrumVol}, \ref{thm:CycleEasy}, \ref{thm:nonalt}.

\begin{lemma}\label{lem:AdditiveVolume}
Suppose that $n$ belted tangles $B(T_1), \hdots B(T_n)$ are hyperbolic with finite volume. Then the belted sum $B(T_1) + \hdots + B(T_n)$ is also a hyperbolic link and has volume $\emph{vol}(B(T_1)) + \hdots + \emph{vol}(B(T_n))$.
\end{lemma}

\begin{figure}
\def\svgwidth{250pt}
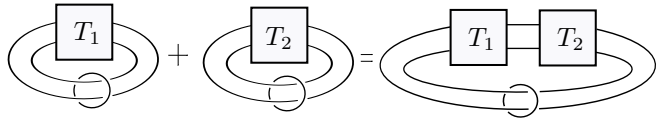
\caption{The belted sum of two tangles.}
\label{fig:BeltedSum}
\end{figure} 

In addition to the belted sum of tangles, we can sum tangles in the manner described below. 
\begin{definition}
Let $T_1$ and $T_2$ be tangle diagrams. One may connect these tangle diagrams as indicated in Figure \ref{fig:ConwaySum} to form the tangle $T_1 + T_2$. We call $T_1 + T_2$ the \textit{tangle sum}, or \textit{Conway sum} of $T_1$ and $T_2$.
\end{definition}

\begin{figure}
\def\svgwidth{250pt}
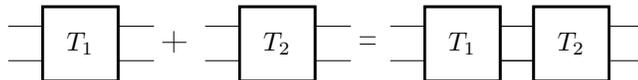
\caption{The Conway sum of two tangles. The tangle $T_1 + T_2$ is pictured on the right.}
\label{fig:ConwaySum}
\end{figure}

Lemma \ref{lem:AdditiveVolume} showed that volume is additive under belted sums of tangles. Similarly, we have the following lemma describing the determinant of a Conway sum of tangles. 

\begin{lemma}\label{lem:MultiplicativeDeterminant}
Let $T_1, \hdots, T_n$ be tangles and let $T = T_1 + \hdots + T_n$ be the tangle sum. Then $\det(D(T)) = \det(D(T_1))\hdots\det(D(T_n))$.
\end{lemma}
\begin{proof}
We prove the case $n = 2$, and the result then follows by induction. By  \cite[Theorem 4.9]{Kauffman}, $\nabla_{D(T_1 + T_2)} = \nabla_{D(T_1)}\nabla_{D(T_2)}$ where $\nabla$ denotes the Conway polynomial. The Conway polynomial $\nabla_L$ of a link $L$ is related to the Alexander polynomial $\Delta_L$ of $L$ via $\Delta_L(x^2) = \nabla_L(x-x\inv)$. Now $\det(L) = |\Delta_L(-1)|$, and it follows that $\det(D(T_1+T_2)) = \det(D(T_1))\det(D(T_2))$.
\end{proof} 

We now recall a result of Futer, Kalfagianni, and Purcell \cite{FKP} which will play an important role in the proofs of Theorems \ref{thm:CycleEasy}, \ref{thm:nonalt}, and \ref{thm:spectrum}. We begin with the following definition. A tangle $T$ is an \textit{east-west} twist if $N(T)$ is the standard diagram of a $(2,q)$-torus link. To simplify notation throughout, we define $\xi_n$ as follows:
\begin{equation}\label{def:xi}
\xi_n := \left( 1 - \left(\dfrac{8 \pi}{11.524 + n \sqrt[4]2} \right)^2\right)^{3/2}
\end{equation}
Note that $\displaystyle \lim_{n \to \infty} \xi_n = 1$. We now state the following theorem which follows directly from the proof of \cite[Theorem 1.5]{FKP}.

\begin{theorem}\label{thm:11.524}
Let $T_1, \hdots, T_n$, $n\geq 12$, be tangles admitting prime, alternating diagrams, none of which is an east-west twist. Let $K$ be a knot or link which can be written as the closure of the Conway sum $N(T_1+\hdots+T_n)$. Let $L$ be the belted sum $B(T_1)+ \hdots + B(T_n)$. Then $K$ is hyperbolic and 
\begin{equation}\label{eqn:11.524}
\textup{vol}(S^3 \backslash K) \geq  \xi_n\textup{vol}(S^3 \backslash L)
\end{equation}
\end{theorem}

We record here the following well-known theorem of W. Thurston which may be found in \cite[Theorem 3.4]{ThurstonNotes}.

\begin{theorem}\label{thm:VolumeDrops}
Suppose that $M_1$ is a complete, finite-volume, hyperbolic manifold and that $M_2 \neq M_1$ is a complete hyperbolic manifold obtained from $M_1$ by Dehn filling along one of the cusps. Then $\text{vol}(M_1) > \text{vol}(M_2)$.
\end{theorem}

\subsection{Adequate Tangles}
We state some results about adequate link diagrams that will play an important role in the proofs of Theorems \ref{thm:CycleEasy}, \ref{thm:nonalt}, \ref{thm:spectrum}, and \ref{thm:spectrum2}. For the definition of an adequate link diagram and nugatory crossings, see \cite[Definition 5.2]{Lickorish}.

\begin{definition}
A link diagram $D \subset S^2$ is \textit{prime} if any simple closed curve in $S^2$ that meets $D$ transversely at two points bounds, on one side of it, a disk that intersects $D$ in a diagram $U$ of the unknotted ball-arc pair. The diagram $D$ is \textit{strongly prime} if, in addition, $U$ is the zero-crossing diagram. 
\end{definition}

We will use the following theorem, which may be found in \cite[Corollary 3.21]{Guts} to show that a link diagram is prime.

\begin{theorem}\label{thm:prime}
Suppose that $K$ is a non-split, prime link.
Then every adequate diagram of $K$ without nugatory crossings is prime.
\end{theorem}

The following proposition gives a way to show that a link is non-alternating. We will use this in the proof of Theorem \ref{thm:nonalt} where we show that there exist geometrically maximal sequences of non-alternating links. 

\begin{proposition}\label{prop:AdequateFacts}
Let $L$ be a link and $D$ a diagram for $L$ having $n$ crossings. Let $\emph{breadth}(L)$ be the breadth of the Jones polynomial for $L$, in other words, the difference between the maximal and minimal degrees of the Jones polynomial for $L$. 
\begin{enumerate}
\item If $D$ is reduced and alternating, then $D$ is adequate.
\item If $D$ is an adequate diagram, then $c(L) = n$.
\item If $D$ is reduced and alternating, then $\emph{breadth}(L) = n$.
\item If $D$ is non-alternating and prime, then $\emph{breadth}(L) < n$.
\item If $L$ has a prime, non-alternating, adequate diagram, then $L$ is non-alternating.
\end{enumerate}
\end{proposition}
\begin{proof}
Proofs of parts (1) through (4) may be found in \cite[Proposition 5.3, Theorem 5.9, and Corollary 5.14]{Lickorish}. 

Suppose $D_1$ is a prime, non-alternating, adequate diagram for $L$, and that $D_2$ is an alternating diagram for $L$. By removing nugatory crossings, we may assume that $D_2$ is a reduced, alternating diagram. Therefore both $D_1$ and $D_2$ are adequate diagrams and have the same number of crossings by part (2) of the proposition. Let $n$ be the number of crossings in $D_1$ (or $D_2$). Since $D_1$ is non-alternating, part (4) of the proposition implies $\text{breadth}(L) <n$. On the other hand, since $D_2$ is reduced and alternating, part (3) implies that $\text{breadth}(L) = n$, a contradiction.
\end{proof}

\begin{definition}
A tangle diagram $T$ is \textit{adequate} if the diagrams of both $N(T)$ and $D(T)$ are adequate.  In the event that both closures produce reduced, alternating diagrams, the tangle $T$ is said to be \textit{strongly alternating}. A strongly alternating tangle is adequate.
\end{definition}

\begin{proposition}\label{prop:adequate}
The tangle sum of $n$ adequate tangles is adequate for $n \geq 2$.
\end{proposition}
\begin{proof}
It follows from work of Lickorish and Thistlethwaite \cite[Proposition 4]{Thistlethwaite} that the tangle sum of two adequate tangles is adequate. The proposition readily follows by induction on $n$.
\end{proof}

We close this section with the following corollary showing that crossing number of adequate tangles is additive. This plays a key role in the proof of Theorems \ref{thm:IntroSpectrumVol} and \ref{thm:IntroSpectrumDet}.

\begin{corollary}\label{cor:additivity}
Let $T_1, \hdots, T_n$ be adequate tangle diagrams and let $c(N(T_i))$ be the crossing number of the closure for $i = 1,\hdots n$. Let $T_1 + \hdots + T_n$ be the tangle sum. Then the crossing numbers satisfy 
\begin{align*}
c(N(T_1 + \hdots + T_n)) & = c(N(T_1)) + \hdots + c(N(T_n))\\
c(D(T_1 + \hdots + T_n)) & = c(D(T_1)) + \hdots + c(D(T_n))
\end{align*}
\end{corollary}
\begin{proof}
The natural diagram representing the link $N(T_1 + \hdots + T_n)$ as a Conway sum has $c(N(T_1)) + \hdots + c(N(T_n))$ crossings. Since $T_1, \hdots, T_n$ are adequate, we know that $N(T_1 + \hdots + T_n)$ is adequate by Proposition \ref{prop:adequate}. Note that an adequate diagram of a link has the minimal number of crossings by Proposition \ref{prop:AdequateFacts}. Therefore $c(N(T_1 + \hdots + T_n)) = c(N(T_1)) + \hdots + c(N(T_n))$. The proof that $c(D(T_1 + \hdots + T_n)) = c(D(T_1)) + \hdots + c(D(T_n))$ is similar.
\end{proof}

\section{Non-alternating Geometrically Maximal Knots}

We now address the question of whether all the conditions in Theorem \ref{thm:CKPthm} are necessary. We begin with a construction of geometrically maximal sequence of links for which each link contains a cycle of tangles. 

\begin{theorem}\label{thm:CycleEasy}
Let $\{T_i\}_{i = 1}^\infty$ be a sequence of knotty tangles admitting prime, alternating diagrams, none of which is an east-west twist. Let $K_i = N(T_i)$ for all $i$. Suppose that each $K_i$ is hyperbolic and that the sequence $\{K_i\}_{i = 1}^\infty$ is geometrically maximal. For each $n \in \N$, let $T(n)$ be the Conway sum of $n$ copies of $T_n$, and define $K(n) = N(T(n))$. Then $K(n)$ is hyperbolic whenever $n \geq 12$, and the sequence $\{K(n)\}_{n = 1}^\infty$ is a geometrically maximal sequence of links, all but one of which contains a cycle of tangles.
\end{theorem}
\begin{proof}
By construction, whenever $n \geq 2$, we have that $K(n)$ contains a cycle of tangles. Moreover, whenever $n \geq 12$ we have from Theorem \ref{thm:CKPthm} that $K(n)$ is hyperbolic. We now show that the sequence $\{K(n)\}_{n = 1}^\infty$ is geometrically maximal. Let $L_i = B(T_i)$, the belted link corresponding to $L_i$, and for each $n \in \N$ let $L(n)$ be the belted sum of $n$ copies of $L_n$. By Theorem \ref{thm:11.524} we have that $\text{vol}(K(n)) \geq \xi_n \text{vol}(L(n))$ where $\xi_n$ was defined in (\ref{def:xi}). We obtain
\begin{align}
\notag \text{vol}(L(n)) & = \sum_{i = 1}^n \text{vol}(L_n) && \text{by Lemma \ref{lem:AdditiveVolume}}\\
\notag & = n\, \text{vol}(L_n)\\
\label{eqn:CycleEasy1} & > n\, \text{vol}(K_n) && \text{by Theorem \ref{thm:VolumeDrops}}
\end{align}
Observe that $c(K(n)) \leq n\, c(K_n)$ since the natural diagram demonstrating $K(n)$ as a Conway sum has $n \, c(K_n)$ crossings. Therefore
\begin{align}
\notag \dfrac{\text{vol}(K(n))}{c(K(n))} & \geq \dfrac{\text{vol}(K(n))}{n\, c(K_n)}\\
\notag & \geq \xi_n\, \dfrac{\text{vol}(L(n))}{n\, c(K_n)} && \text{ by Theorem \ref{thm:11.524}}\\
\label{eqn:CycleEasy2} & > \xi_n\, \dfrac{n\, \text{vol}(K_n)}{n\, c(K_n)} && \text{ by (\ref{eqn:CycleEasy1})}
\end{align}
Note that $\lim_{n \to \infty} \xi_n = 1$. Since $\{K_n\}_{n = 1}^\infty$ is geometrically maximal, we know that
$$\lim_{n \to \infty} \dfrac{\text{vol}(K_n)}{c(K_n)} = v_8$$
hence by (\ref{eqn:CycleEasy2}) we obtain
$$\lim_{n \to \infty} d_{\text{vol}}(K(n)) = \lim_{n \to \infty} \frac{\text{vol}(K(n))}{c(K(n))} \geq \lim_{n \to \infty} \xi_n \, \frac{ \text{vol}(K_n)}{c(K_n)} = v_8$$
implying $\{K(n)\}_{n = 1}^\infty$ is geometrically maximal.
\end{proof}

We can use a method similar to the proof of Theorem \ref{thm:CycleEasy} to prove that there exist geometrically maximal sequences of non-alternating links. We will construct a specific example of such a sequence. We begin by considering a family of weaving tangles in Definition \ref{def:WeavingTangle}. By taking a Conway sum of these weaving tangles and their reflections (i.e. the result of changing all the over-crossings to under-crossings and vice versa), we will obtain a geometrically maximal sequence of non-alternating links.

\begin{definition}\label{def:WeavingTangle}
Let $B_m$ be the braid group on $m$ strings. For $i = 1,\hdots m-1$, let $\sigma_i \in B_m$ correspond to twisting the $i$th strand under the $(i+1)$st strand. Let $L_{m,n}$ be the element $(\sigma_1 \sigma_2\inv\sigma_3 \sigma_4\inv\hdots \sigma_{m-1}^{(-1)^{m-1}})^n \in B_m$ depicted on the left hand side of Figure \ref{fig:WeavingTangle}. Let $W_{m,n}$ be the result of closing all but the second and third strands of the braid as in the right hand side of Figure \ref{fig:WeavingTangle}. We will call $W_{m,n}$ the \textit{weaving tangle} of order $n$ on $m$ strands. The knots $N(W_{m,n})$ are the weaving knots of \cite{CKP}.
\end{definition}

\begin{figure}
\def\svgwidth{250pt}
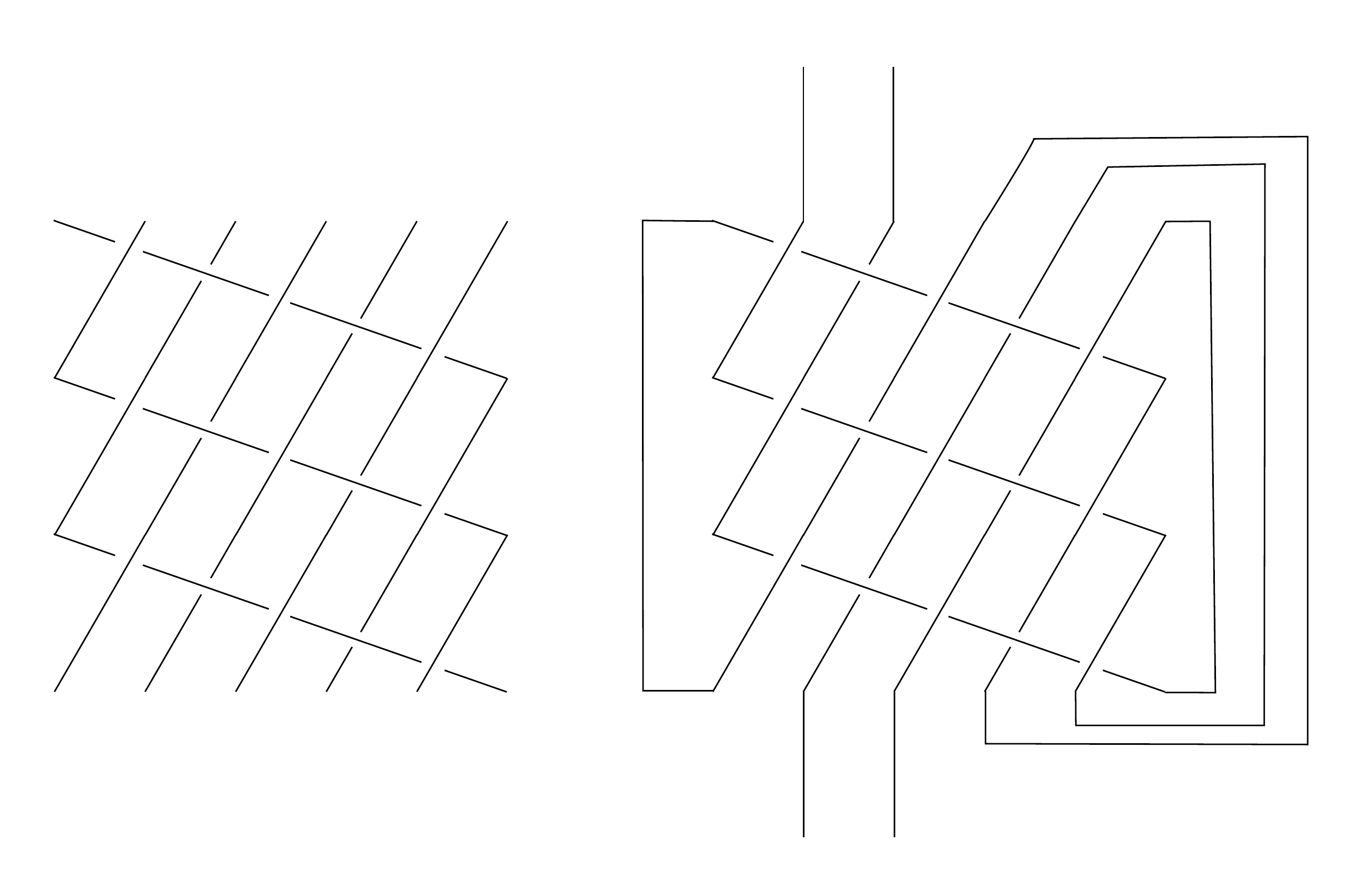
\caption{Left: The braid element $L_{6,3}$. Right: The weaving tangle $W_{6,3}$.}
\label{fig:WeavingTangle}
\end{figure}

\begin{figure}
\def\svgwidth{450pt}
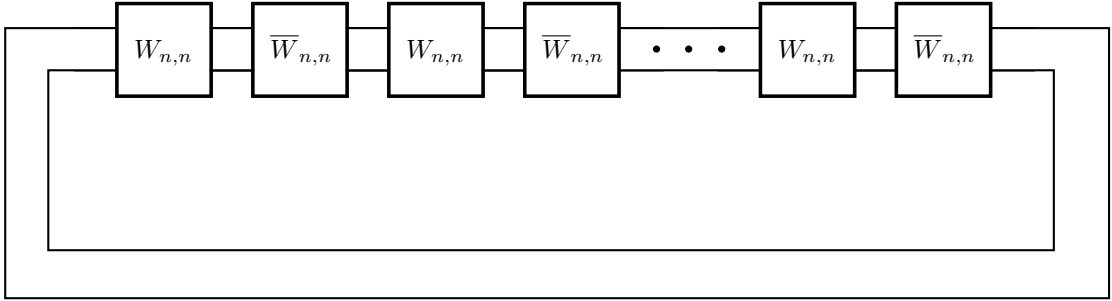
\caption{The link $K(n)$ for the proof of Theorem \ref{thm:nonalt}. It contains $n$ copies of $W_{n,n}$ and $n$ copies of $\overline{W}_{n,n}$.}
\label{fig:Nonalternating}
\end{figure}

\begin{theorem} \label{thm:nonalt}
There exist geometrically maximal sequences of non-alternating links.
\end{theorem}
\begin{proof}
Let $W_{k,k}$ be the weaving tangle and let $\overline{W}_{k,k}$ be its reflection, i.e. the result of changing each over-crossing to an under-crossing and vice versa. For $n \geq 1$ define $$K(n) = N(W_{n,n} + \overline{W}_{n,n} + \hdots + W_{n,n} + \overline{W}_{n,n})$$ to be the closure of the Conway sum of $n$ copies of $W_{n,n}$ with $n$ copies of $\overline{W}_{n,n}$ as indicated in Figure \ref{fig:Nonalternating}. Let $K_n = N(W_{n,n})$ and $\overline{K}_n = N(\overline{W}_{n,n})$, and $L_n = B(W_{n,n})$ and $\overline{L}_n = B(\overline{W}_{n,n})$.

We prove that the sequence $\{K(n)\}_{n = 1}^\infty$ is geometrically maximal. First, observe that when $n \geq 6$, Theorem \ref{thm:11.524} implies that $K(n)$ is hyperbolic. Let $$L(n)=B(W_{n,n}) + B(\overline{W}_{n,n}) + \hdots + B(W_{n,n}) + B(\overline{W}_{n,n})$$ be the belted sum of $n$ copies of $B(W_{n,n})$ with $n$ copies of $B(\overline{W}_{n,n})$. Similar to the proof of Theorem \ref{thm:CycleEasy} we have
\begin{align}
\notag \text{vol}(L(n)) & = \sum_{i = 1}^n (\text{vol}(L_n) + \text{vol}(\overline{L}_n)) && \text{by Lemma \ref{lem:AdditiveVolume}}\\
\notag & = n[\text{vol}(L_n) + \text{vol}(\overline{L}_n)]\\
\label{eqn:nonalt1} & > n[\text{vol}(K_n) + \text{vol}(\overline{K}_n)] && \text{by Theorem \ref{thm:VolumeDrops}}
\end{align}
Since $K_n$ and $\overline{K_n}$ are reduced and alternating, by counting the number of crossings in Figure \ref{fig:WeavingTangle} and using Proposition \ref{prop:AdequateFacts}, we find that $c(K_n) = c(\overline{K}_n) = n(n-1)$.  Now the diagram of $K(n)$ depicted in Figure \ref{fig:Nonalternating} has $2n^2(n-1)$ crossings, implying $c(K(n)) \leq 2n^2(n-1)$.  Therefore
\begin{align}
\notag \frac{\text{vol}(K(n))}{c(K(n))} & \geq \frac{\text{vol}(K(n))}{2n^2(n-1)}\\
\notag & \geq \xi_{2n} \frac{\text{vol}(L(n))}{2n^2(n-1)} && \text{by Theorem \ref{thm:11.524}}\\
\notag & > \xi_{2n} \frac{n[\text{vol}(K_n) + \text{vol}(\overline{K}_n)]}{2n^2(n-1)} && \text{by (\ref{eqn:nonalt1})} \\
& = \frac{\xi_{2n}}{2} \left(\frac{\text{vol}(K_n)}{c(K_n)} + \frac{\text{vol}(\overline{K}_n)}{c(\overline{K}_n)} \right)\\
\label{eqn:nonalt2}  & = \frac{\xi_{2n}}{2} [d(K_n) + d(\overline{K}_n)]
\end{align} 
Both $K_n$ and $\overline{K}_n$ are geometrically maximal by Theorem \ref{thm:CKPthm}, so $\lim_{n \to \infty} d(K_n) = \lim_{n \to \infty} d(\overline{K}_n) = v_8$. Moreover, $\lim_{n \to \infty} \xi_{2n} = 1$, so (\ref{eqn:nonalt2}) implies that
$$\lim_{n \to \infty} d(K(n)) = v_8.$$




We show that $K(n)$ is non-alternating. Let $D$ be the diagram shown in Figure \ref{fig:Nonalternating}. Observe that $D$ is non-alternating. Since $W_{k,k}$ and $\overline{W}_{k,k}$ are strongly alternating (hence adequate) tangle diagrams, Proposition \ref{prop:adequate} implies that $D$ is adequate. Since an adequate diagram has the minimal number of crossings, there are no nugatory crossings in $D$. An adequate (hence minimal crossing) diagram of a split link must also be split. Since $D$ is non-split, it follows that $K(n)$ is non-split. Now $K(n)$ is hyperbolic when $n \geq 6$, as noted above. Therefore $K(n)$ is a prime link. It now follows from Theorem \ref{thm:prime} that the diagram $D$ is prime. Proposition \ref{prop:AdequateFacts} now implies that $K(n)$ is non-alternating. 
\end{proof}

\noindent \textit{Remark}: The sequence of links $\{K(n)\}_{n = 1}^\infty$ in the proof of Theorem \ref{thm:nonalt} is geometrically maximal and each link of the sequence contains a cycle of tangles.

\section{The Spectrum of Volume Densities} \label{sec:Spectrum}

We now turn to the proof that the set of volume densities is dense in $[0,v_8]$. In Lemma \ref{lem:MinimalKnot}, we prove the existence of tangles whose closures form a \textit{geometrically minimal} sequence of links, i.e. a sequence links with volume density approaching $0$. On the other hand, Lemma \ref{lem:MaximalTangle} produces a sequence of tangles whose closures form a geometrically maximal sequence of links. 

The idea of the proof of Theorem \ref{thm:spectrum} is to combine geometrically maximal tangles with geometrically minimal tangles via Conway sum. Key to the proof is the fact that hyperbolic volume is additive under belted sum (see Lemma \ref{lem:AdditiveVolume}), and that the crossing number of adequate tangles is additive under Conway sum (see Corollary \ref{cor:additivity}). The ratio of geometrically maximal to geometrically minimal tangles may then be controlled so that one may find a sequence of links with volume density approaching any number in the interval $[0, v_8]$.

\begin{figure}
\def\svgwidth{250pt}
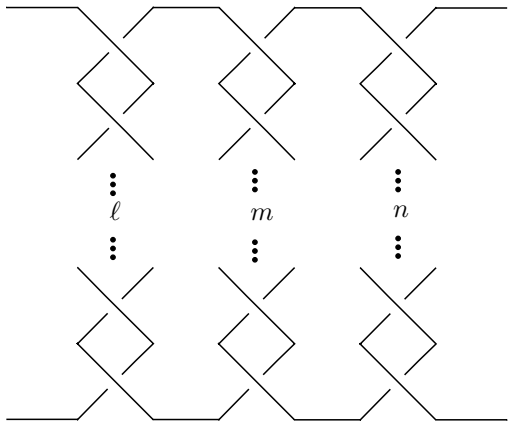
\caption{The pretzel tangle $P_{\ell,m,n}$ has $\ell$, $m$, and $n$ crossings in the respective twist regions.}
\label{fig:PretzelTangle}
\end{figure}

\begin{lemma}\label{lem:MinimalKnot}
There exists a sequence of strongly alternating tangles $\{T_m\}_{m = 1}^\infty$ such that $$\displaystyle \lim_{m \to \infty} d_{\emph{vol}}(N(T_m)) = 0$$ and $N(T_m)$ is hyperbolic for all $m$.
\end{lemma}
\begin{proof}
Let $P_{\ell,m,n}$ be the pretzel tangle shown in Figure \ref{fig:PretzelTangle}. It is well-known that many pretzel links are hyperbolic. For example, whenever $\ell,m,n \geq 7$, work of Futer, Kalfagianni, and Purcell \cite[Theorem 1.2]{FKPtwists} implies that $N(P_{\ell,m,n})$ is hyperbolic. Moreover, $P_{\ell,m,n}$ is strongly alternating whenever $\ell, m$, and $n \geq 2$. Fix $\ell\geq 7$ and $n \geq 7$ and let $K_m = N(P_{\ell,m,n})$. By work of Lackenby, Agol, and D. Thurston \cite{LackenbyBound}, we know that $\text{vol}(K_m) \leq 10 v_3(\text{tw}(K_m)-1)$, where $\text{tw}(K_m)$ is the number of twist regions in a diagram of $K_m$. In particular, $\text{tw}(K_m) = 3$ for all values of $m$ so $\text{vol}(K_m) \leq 20 v_3$. It follows that 
$$ \lim_{m \to \infty} d_{\text{vol}}(K_m) = \lim_{m \to \infty} \frac{\text{vol}(K_m)}{c(K_m)} \leq \lim_{m \to \infty} \frac{20 v_3}{\ell + m + n} = 0$$ 
\end{proof}

\begin{lemma}\label{lem:MaximalTangle}
There exists a sequence of strongly alternating (hence adequate) tangles $\{T_n\}_{n =1}^\infty$ such that $\{N(T_n)\}_{n = 1}^\infty$ is geometrically maximal.
\end{lemma}
\begin{proof}
Consider the weaving tangles $W_{m,n}$ from Definition \ref{def:WeavingTangle} and depicted in Figure \ref{fig:WeavingTangle}. Let $T_n = W_{n,n}$. It was shown in \cite[Theorem 1.9]{CKP} that $\{N(T_n)\}_{n = 1}^\infty$ is a geometrically maximal sequence of links. Note that $N(T_n)$ and $D(T_n)$ have reduced, alternating diagrams, whenever $n \geq 4$.
\end{proof}

We have established examples of geometrically minimal and geometrically maximal sequences of adequate links. We now use these to prove the following theorem, which is identical to Theorem \ref{thm:IntroSpectrumVol}.

\begin{theorem}\label{thm:spectrum}
The set $\mathcal{C}_{\emph{vol}}$ of volume densities of hyperbolic links is a dense subset of $[0,v_8]$, and $\emph{Spec}_{\emph{vol}} = [0,v_8]$. In other words, given $x \in [0,v_8]$ there exists a sequence $\{K_n\}_{n = 1}^\infty$ of hyperbolic links such that the volume densities satisfy $\displaystyle\lim_{n \to \infty} d_{\emph{vol}}(K_n) = x$.
\end{theorem}
\begin{proof}
It suffices to show that for any non-negative integers $a$ and $b$ with $a$ and $b$ not both zero, and any $\epsilon > 0$, there exists a link $K$ such that 
$$ \frac{b}{a+b} v_8 - \epsilon < d_{\text{vol}}(K) < \frac{b}{a+b}v_8 + \epsilon.$$
Let $\epsilon > 0$ be arbitrary. Let  $P_{\ell,m,n}$ be the pretzel tangles from Lemma \ref{lem:MinimalKnot}. Let $\widetilde{T}_1 = P_{7,\widetilde{m},7}$ and $\widetilde{K}_1 = N(\widetilde{T}_1)$, where $\widetilde{m}$ is chosen to be large enough that
\begin{equation}\label{size_of_m}
\frac{ 40\, a \, v_3}{(a + b)(\widetilde{m} + 14)} < \frac{\epsilon}{2}
\end{equation}
In Lemma \ref{lem:MaximalTangle}, it was pointed out that the sequence of links $\{N(W_{k,k})\}_{k = 1}^\infty$, where $W_{k,k}$ is the weaving tangle, is a geometrically maximal sequence of links. Therefore we may choose $T_2 = W_{k,k}$ and $K_2 = N(T_2)$ with $k$ sufficiently large that 
\begin{equation}\label{maximal_size}
\frac{\vol[K_2]}{c(K_2)} > v_8 - \left(\frac{a+b}{b}\right)\frac{\epsilon}{2} \quad \text{ and } \quad c(K_2) \geq c(\widetilde{K}_1)
\end{equation}
Since $\widetilde{K}_1$ and $K_2$ have reduced alternating diagrams, by counting the number of crossings in these diagrams, we see that 
\begin{equation}
k(k-1) = c(K_2) \geq c(\widetilde{K}_1) = \widetilde{m} + 14.
\end{equation}
Let $m = k(k-1) - 14$. Define $T_1 = P_{7,m,7}$ and $K_1 = N(T_1)$. This choice of $m$ implies that
\begin{equation}
c(K_1) = c(K_2)
\end{equation}
Observe that $m \geq \widetilde{m}$, $\text{tw}(K_1) = 3$ and $c(K_1) = m + 14$, so (\ref{size_of_m}) guarantees that
\begin{equation}\label{minimal_size}
\frac{10\,a\, v_3(\text{tw}(K_1) + 1)}{(a+b)\,c(K_1)} = \frac{ 40\, a \, v_3}{(a + b)(m + 14)}\leq  \frac{ 40\, a \, v_3}{(a + b)(\widetilde{m} + 14)} < \frac{\epsilon}{2}
\end{equation}
Choose $n \geq 12$ so that $\xi_n$ as defined in (\ref{def:xi}) satisfies
\begin{equation}\label{n_size}
\xi_n > \max \left\{ 1-  \frac{(a + b)\epsilon}{2\, b\,v_8 -\epsilon(a+b)}, \frac{2\,b\,v_8}{2\,v_8 + \epsilon(a+b)} \right\}
\end{equation}
Since $0 < \xi_n < 1$, we need to check that such a choice of $n$ is possible. Since $a$, $b$, $v_8$, and $\epsilon$ are positive we know that $$0 < \frac{2bv_8}{2bv_8 + \epsilon(a+b)} < 1$$
Moreover, as long as $\epsilon < 2\,\,v_8/(a+b)$ we have $2\,b\,v_8 > \epsilon(a+b)$, hence $$1-  \frac{(a + b)\,\epsilon}{2\, b\,v_8 -\epsilon(a+b)}< 1$$
Let $L_1 = B(T_1)$ and $L_2 = B(T_2)$ be the belted links corresponding to $T_1$ and $T_2$ respectively.  Form the link $L$ by taking the belted sum of $a\cdot n$ copies of $L_1$ with $b\cdot n$ copies of $L_2$. Let $K$ be the result of filling the belt of $L$ via the meridional Dehn filling. Note that $K$ is simply the result of taking the closure of the tangle sum of $a\cdot n$ copies of $T_1$ with $b \cdot n$ copies of $T_2$ as shown in Figure \ref{fig:VolumeSpectrum}. 
\begin{figure}
\def\svgwidth{450pt}
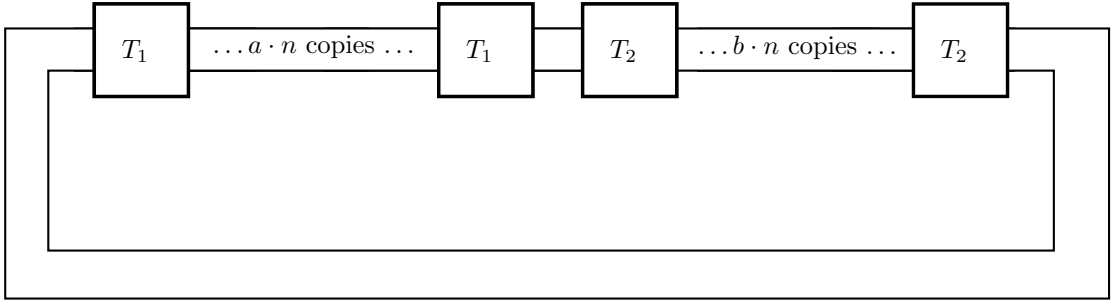
\caption{The link $K$ is the closure of the tangle sum of $a \cdot n$ copies of $T_1$ with $b \cdot n$ copies of $T_2$ shown above.}
\label{fig:VolumeSpectrum}
\end{figure}
We obtain
\begin{align}
\vol[K] & \geq \xi_n \vol[L] && \text{by Theorem \ref{thm:11.524} }\\
& = \xi_n(a\cdot n \cdot \vol[L_1] + b\cdot n \cdot \vol[L_2])&& \text{by Lemma \ref{lem:AdditiveVolume}} \\
\label{eqn:lower} & > \xi_n(a\cdot n \cdot \vol[K_1] + b\cdot n \cdot \vol[K_2]) && \text{by Theorem \ref{thm:VolumeDrops}}
\end{align}
Since $T_1$ and $T_2$ are strongly alternating tangles, it follows from Corollary \ref{cor:additivity} that 
\begin{equation}\label{eqn:xing}
c(K) = n\cdot a \cdot c(K_1) + n\cdot b \cdot c(K_2)
\end{equation} This produces the following inequalities to form a lower bound for the density $d_{\text{vol}}(K)$.
\begin{align*}
d(K)_{\text{vol}} = \frac{\vol[K]}{c(K)} & > \frac{\xi_n( a\cdot n \cdot \vol[K_1] + b\cdot n \cdot \vol[K_2])}{c(K)} && \text{by (\ref{eqn:lower})}\\
& = \frac{\xi_n( a\cdot n\cdot \vol[K_1] + b\cdot n\cdot \vol[K_2])}{a\cdot n\cdot c(K_1) + b\cdot n\cdot c(K_2)} && \text{by (\ref{eqn:xing})}\\
& = \xi_n \frac{\vol[K_1]}{c(K_1)} \frac{a}{a+b} + \xi_n \frac{\vol[K_2]}{c(K_2)}\frac{b}{a+b} && \text{since } c(K_1) = c(K_2)\\
& \geq \xi_n \frac{\text{vol}(K_2)}{c(K_2)}\frac{b}{a+b}\\
& > \left(1 - \frac{\epsilon(a+b)}{2\,b\, v_8 - \epsilon(a+b)} \right) \left( v_8 - \frac{a+b}{b}\frac{\epsilon}{2} \right)\frac{b}{a+b} && \text{by (\ref{maximal_size}) and (\ref{n_size}}) \\
\intertext{Which after some algebra simplifies to}
& = \left(1 - \frac{\epsilon(a+b)}{2\,b\, v_8 - \epsilon(a+b)}  \right) \left( v_8 \frac{b}{a+b} - \frac{\epsilon}{2}\right)\\
& = v_8 \frac{b}{a+b} - \frac{\epsilon\, v_8\, b}{2\, b\, v_8 - \epsilon(a+b)} + \frac{\epsilon^2 (a+b)}{2(2\,b\,v_8 - \epsilon(a+b))} - \frac{\epsilon}{2}\\
& = v_8 \frac{b}{a+b} + \frac{\epsilon}{2} \left( \frac{-2\,v_8\, b}{2\,b\,v_8 - \epsilon(a+b)} + \frac{\epsilon(a+b)}{2\,b\,v_8 - \epsilon(a+b)}\right) - \frac{\epsilon}{2}\\
& = v_8 \frac{b}{a+b} - \epsilon
\end{align*}

We now find an upper bound for $d_{\text{vol}}(K)$. Since $K$ is obtained from $L$ by Dehn filling, Theorem \ref{thm:VolumeDrops} implies that $\text{vol}(K) < \text{vol}(L)$. Moreover, Lemma \ref{lem:AdditiveVolume} implies that $\text{vol}(L) = a \cdot n \cdot \text{vol}(L_1) + b \cdot n \cdot\text{vol}(L_2)$. Therefore  
\begin{equation}\label{SplitInTwo}
\vol[K] < a\cdot n\cdot \vol[L_1] + b\cdot n\cdot \vol[L_2]
\end{equation}
It was shown by Lackenby, Agol, and D. Thurston \cite{LackenbyBound} that $a\cdot n\cdot \vol[L_1]< 10\, a\cdot n\cdot v_3(\text{tw}(L_1) - 1)$. Consider the diagrams of $L_1$ and $K_1$. Adding the belt circle to $K_1$ to form $L_1$ adds no more than two twist regions, as indicated in Figure \ref{fig:BeltCircleTwistRegions}. 
\begin{figure}
\def\svgwidth{250pt}
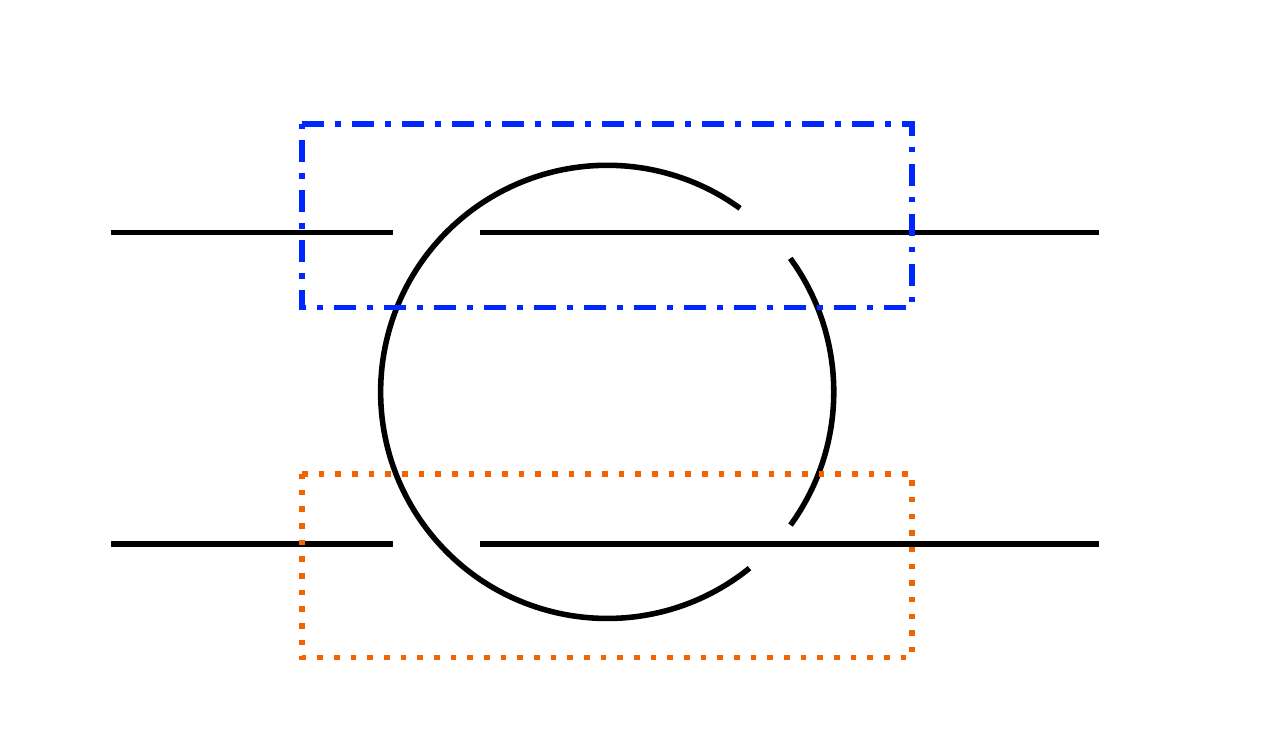
\caption{The belt circle adds no more than the two twist regions indicated in the figure.}
\label{fig:BeltCircleTwistRegions}
\end{figure}
Therefore $\text{tw}(L_1) \leq \text{tw}(K_1) + 2$, hence 
\begin{equation}\label{LeftUpperBound}
a\cdot n\cdot \vol[L_1] < 10\, a\, n\, v_3 (\text{tw}(K_1) + 1)
\end{equation}
Define $L_2^n$ to be the belted sum of $n$ copies of $L_2$ and let $K_2^n$ be the result of filling the belt of $L_2^n$ via the meridional filling. Then $K_2^n = N(T_2 + \hdots + T_2)$ where $+$ denotes Conway sum and the sum consists of $n$ copies of $T_2$ (see Figure \ref{fig:K2n}).
\begin{figure}
\def\svgwidth{250pt}
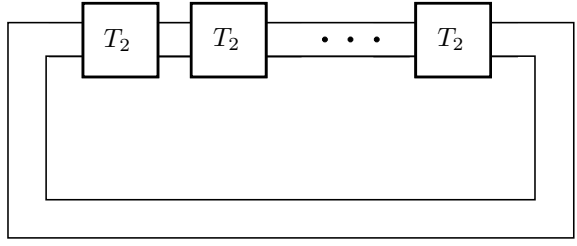
\caption{The link $K_2^n$ is the closure of the sum of $n$ copies of $T_2$ as indicated above.}
\label{fig:K2n}
\end{figure}
Then Lemma \ref{lem:AdditiveVolume} implies $\text{vol}(L_2^n) = n \, \text{vol}(L_2)$. Theorem \ref{thm:11.524} implies $\text{vol}(L_2^n) \leq \text{vol}(K_2^n)/\xi_n$. Therefore
 \begin{equation}\label{eqn:summary1}
b\cdot n\cdot \vol[L_2] 
= b \cdot \vol[L_2^n]
\leq \frac{b}{\xi_n}\vol[K_2^n]
\end{equation}
Since $v_8$ forms an upper bound on volume density, we know that $\text{vol}(K_2^n) \leq v_8\, c(K_2^n) $. It follows that 
\begin{equation}\label{eqn:summary2}
\frac{b}{\xi_n}\,\vol[K_2^n] \leq \dfrac{b}{\xi_n}\, v_8\, c(K_2^n)
\end{equation}
The diagram for $K_2^n$ shown in Figure \ref{fig:K2n} has $n\cdot c(K_2)$ crossings. Therefore $c(K_2^n) \leq n\cdot c(K_2)$. This yields
\begin{equation}\label{eqn:summary3}
\dfrac{b}{\xi_n}\, v_8\, c(K_2^n) \leq \frac{b\, n\, v_8}{\xi_n}\,c(K_2)
\end{equation}
Summarizing (\ref{eqn:summary1}), (\ref{eqn:summary2}), and (\ref{eqn:summary3}) we see that
\begin{equation}\label{RightUpperBound}
b\cdot n\cdot \vol[L_2] \leq \frac{b\, n\, v_8}{\xi_n}c(K_2).
\end{equation}

\noindent This allows us to produce an upper bound on the volume density $d(K)$ as follows.
\begin{align*}
d_{\text{vol}}(K) = \frac{\vol[K]}{c(K)} & < \frac{a\cdot n\cdot \vol[L_1] + b\cdot n\cdot \vol[L_2]}{c(K)} && \text{by (\ref{SplitInTwo})}\\
& = \frac{a\cdot n\cdot \vol[L_1] + b\cdot n\cdot \vol[L_2]}{a \cdot n\cdot c(K_1) + b\cdot n\cdot c(K_2)} && \text{by (\ref{eqn:xing})}\\
& < \frac{10\, a\, n\, v_3(\text{tw}(K_1) + 1) + \frac{b}{\xi_n}(n\, v_8\,  c(K_2))}{a \cdot n\cdot c(K_1) + b\cdot n\cdot c(K_2)} && \text{by (\ref{LeftUpperBound}) and (\ref{RightUpperBound})}\\
& = \frac{10\, a\,v_3(\text{tw}(K_1) + 1)}{(a+b)c(K_1)}+ \frac{b\, v_8\, c(K_2)}{\xi_n(a+b)\,c(K_2)} && \text{since } c(K_1) = c(K_2)\\
& < \frac{\epsilon}{2} + \frac{1}{\xi_n} \frac{bv_8}{a+b} && \text{by (\ref{minimal_size})}\\
& < \frac{\epsilon}{2} + \frac{2bv_8 + \epsilon(a+b)}{2bv_8} \frac{bv_8}{a+b} && \text{by (\ref{n_size})}\\
& = \frac{\epsilon}{2} + \frac{2bv_8 + \epsilon(a+b)}{2(a+b)}\\
& = \frac{b}{a+b}v_8 + \epsilon
\end{align*}
Therefore 
$$\frac{b}{a+b}v_8 - \epsilon < d_{\text{vol}}(K) < \frac{b}{a+b}v_8 + \epsilon$$ as desired.

\end{proof}

\section{The Spectrum of Determinant Densities}

We now turn to studying the spectrum of determinant densities. The proof that $\text{Spec}_{\det}$ contains $[0,v_8]$ will follow a similar method as the proof of Theorem \ref{thm:spectrum}. Namely, we note that there exist sequences of diagrammatically maximal links with determinant density approaching $v_8$, and there exist sequences of \textit{diagrammatically minimal} links, i.e. those with determinant density near 0. We combine these diagrams via Conway sum, but instead of taking the ``$N$" closure as in the proof of Theorem \ref{thm:spectrum}, we consider the ``$D$" closure. We will then use the facts that $\log \det$ is additive under the ``$D$" closure of Conway sums (see Lemma \ref{lem:MultiplicativeDeterminant}), and that the crossing number is additive under Conway sums of adequate diagrams (see Corollary \ref{cor:additivity}). Then by controlling the ratio of diagrammatically minimal links to diagrammatically maximal links, we may obtain a link that has determinant density near any number $x \in [0,v_8]$.
\begin{lemma}\label{lem:DiagramMinimal}
There exists a sequence of adequate tangles $\{T_m\}_{m = 1}^\infty$ such that the sequence of closures $\{D(T_m)\}_{m =1}^\infty$ satisfies $\lim_{n \to \infty} d_{\det}(D(T_n)) = 0$.
\end{lemma}
\begin{proof}
Let $P_{\ell,m,n}$ be the pretzel tangles shown in Figure \ref{fig:PretzelTangle}. Let $T_m = P_{3,m,3}$ and $K_m = D(T_m)$. Then the tangles are strongly alternating, hence adequate. It follows from \cite[Page 100]{Lickorish} (see also work of Dasbach, Kalfagianni, Futer, Lin, and Stoltzfus \cite[Example 4.3]{Dasbach}\footnote{The author thanks Jessica Purcell for pointing out this reference.}) that $\det(K_m) = 6m + 9$. Now
\begin{align*}
d_{\det}(K_m) = \lim_{m \to \infty} \frac{2 \pi \log \det(K_m)}{c(K_m)} = \lim_{m \to \infty} \frac{2 \pi \log(6m + 9)}{m + 6} = 0
\end{align*}
\end{proof}

\begin{lemma}\label{DiagramMaximal}
There exists a sequence of adequate tangles $\{T_m\}_{m = 1}^\infty$ such that the sequence of closures $\{D(T_m)\}_{m =1}^\infty$ satisfies $\lim_{n \to \infty} d_{\det}(D(T_n)) = v_8$.
\end{lemma}
\begin{proof}
Let $W_{m,m}$ be the weaving tangles from Definition \ref{def:WeavingTangle} and depicted in Figure \ref{fig:WeavingTangle}. Then the sequence $\{D(W_{m,m})\}_{m = 1}^\infty$ satisfies the conditions of Theorem \ref{thm:CKPthm2}, hence is diagrammatically maximal.
\end{proof}

Having now established examples of diagrammatically minimal and diagrammatically maximal sequences of links, we can now prove the following theorem, which is identical to Theorem \ref{thm:IntroSpectrumDet}.

\begin{theorem}\label{thm:spectrum2}
The spectrum of determinant densities $\emph{Spec}_{\det}$ contains $[0,v_8]$. In other words, given any $x \in [0,v_8]$ there exists a sequence $\{K_n\}_{n = 1}^\infty$ of links satisfying $\displaystyle \lim_{n \to \infty} d_{\det}(K_n) = x$. 
\end{theorem}
\begin{proof}
It suffices to show that for any non-negative integers $a$ and $b$ with $a$ and $b$ not both zero, and any $\epsilon > 0$, there exists a link $K$ such that
$$\frac{b}{a+b} v_8 - \epsilon < d_{\det}(K) < \frac{b}{a+b}v_8 + \epsilon.$$
Let $\epsilon > 0$ be arbitrary. Let $P_{\ell,m,n}$ be the pretzel tangles. Let $\widetilde{T}_1 = P_{3,\widetilde{m},3}$ and $\widetilde{K}_1 = D(\widetilde{T}_1)$, where $\widetilde{m}$ is chosen large enough that 
\begin{equation}\label{eqn:mdiagram}
\frac{2\pi \log(6m + 9)}{m+6} < \left(\frac{a+b}{a}\right) \frac{\epsilon}{2}
\end{equation}
Since the sequence $\{D(W_{k,k})\}_{k = 1}^\infty$ is diagrammatically maximal, we may choose $T_2 = W_{k,k}$ and $K_2 = D(T_2)$ with $k$ sufficiently large that
\begin{equation}\label{eqn:maximaldiagram}
\left| \frac{2 \pi \log \det(K_2)}{c(K_2)} - v_8 \right|< \left(\frac{a + b}{b}\right) \frac{\epsilon}{2} \quad \text{ and } \quad c(K_2) \geq c(\widetilde{K}_1)
\end{equation}
The choice of $K_2$ ensures that the crossing numbers $c(\widetilde{K}_1)$ and $c(K_2)$ satisfy the inequality
\begin{equation}
k(k-1) = c(K_2) \geq c(\widetilde{K}_1) = \widetilde{m} + 6
\end{equation}
Let $m = k(k-1) - 6$. Define $T_1 = P_{3,m,3}$ and $K_1 = D(T_1)$. This choice of $m$ implies that $$c(K_1) = c(K_2)$$ 
Using the fact that $m \geq \widetilde{m}$ we see that
\begin{equation}\label{eqn:minimaldiagram}
d_{\det}(K_1) = \frac{2\pi \log(6m + 9)}{m+6} \leq \frac{2\pi \log(6\widetilde{m} + 9)}{\widetilde{m}+6} < \left(\frac{a+b}{a}\right) \frac{\epsilon}{2}
\end{equation}
Let $T$ be the tangle sum of $a$ copies of $T_1$ with $b$ copies of $T_2$, and let $K = D(T)$ (see Figure \ref{fig:DeterminantSpectrum}).
\begin{figure}
\centering
\def\svgwidth{450pt}
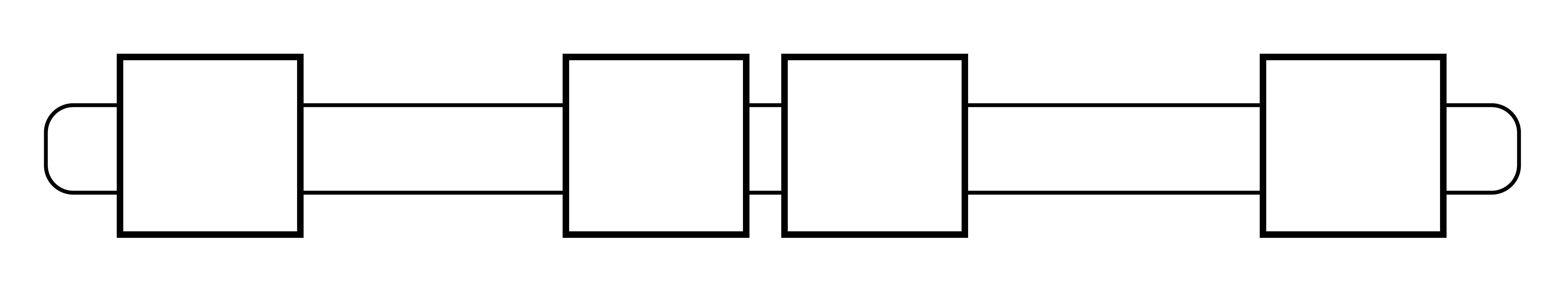
\caption{The link $K$ is the closure of the tangle sum of $a$ copies of $T_1$ with $b$ copies of $T_2$ shown above.}
\label{fig:DeterminantSpectrum}
\end{figure}
Then by Lemma \ref{lem:MultiplicativeDeterminant} we have that $\det(K) = (\det K_1)^a(\det K_2)^b$. Since $T_1$ and $T_2$ are strongly alternating tangles, we know from Corollary \ref{cor:additivity} that $c(K) = a \cdot c(K_1) + b \cdot c(K_2)$. Therefore, we obtain the following equalities:
\begin{align}
\notag d_{\det}(K) & = \frac{2 \pi \log \det(K)}{c(K)}\\
\notag & = \frac{2 \pi \log[(\det K_1)^a(\det K_2)^b]}{a \cdot c(K_1) + b \cdot c(K_2)}\\
\label{eqn:simplerdet} & = \frac{a}{a+b} \frac{2 \pi \log \det (K_1)}{c(K_1)} + \frac{b}{a + b} \frac{ 2 \pi \log \det(K_2)}{c(K_2)} && \text{since } c(K_1) = c(K_2)
\end{align}
This enables us to find the following upper bound:
\begin{align*}
d_{\det}(K) & =  \frac{a}{a+b} \frac{2 \pi \log \det (K_1)}{c(K_1)} + \frac{b}{a + b} \frac{ 2 \pi \log \det(K_2)}{c(K_2)}\\
& < \frac{a}{a+b}\left(\frac{a+b}{a}\right) \frac{\epsilon}{2}  + \frac{b}{a+b} \left[ v_8 + \left(\frac{a+b}{b}\right) \frac{\epsilon}{2} \right] && \text{by (\ref{eqn:maximaldiagram}) and (\ref{eqn:minimaldiagram})} \\
& = \frac{b}{a+b} v_8  + \epsilon
\end{align*}
Now we obtain a lower bound on $d_{\det(K)}:$
\begin{align*}
d_{\det}(K) & = \frac{a}{a+b} \frac{2 \pi \log \det (K_1)}{c(K_1)} + \frac{b}{a + b} \frac{ 2 \pi \log \det(K_2)}{c(K_2)} && \text{by (\ref{eqn:simplerdet})}\\
& \geq \frac{b}{a + b} \frac{ 2 \pi \log \det(K_2)}{c(K_2)}\\	
& > \frac{b}{a + b}\left( v_8 - \left(\frac{a+b}{b}\right) \frac{\epsilon}{2} \right) && \text{ by (\ref{eqn:maximaldiagram})}\\
& > \frac{b}{a + b} v_8 - \epsilon
\end{align*}
\end{proof}

\newpage

\bibliographystyle{abbrv}

\bibliography{Spectra}

\begin{thebibliography}{10}

\bibitem{AdamsBound}
C.~Adams.
\newblock Triple crossing number of knots and links.
\newblock {\em J. Knot Theory Ramifications}, 22(2):1350006, 17, 2013.

\bibitem{Adams}
C.~C. Adams.
\newblock Thrice-punctured spheres in hyperbolic {$3$}-manifolds.
\newblock {\em Trans. Amer. Math. Soc.}, 287(2):645--656, 1985.

\bibitem{qdensity}
A.~Champanerkar, I.~Kofman, and J.~S. Purcell.
\newblock Density spectra for knots.
\newblock {\em arXiv:1506.05841}, 2015.

\bibitem{CKP}
A.~Champanerkar, I.~Kofman, and J.~S. Purcell.
\newblock Geometrically and diagramatically maximal knots.
\newblock {\em arXiv:1411.7915}, 2015.

\bibitem{CKP2}
A.~Champanerkar, I.~Kofman, and J.~S. Purcell.
\newblock Volume bounds for weaving knots.
\newblock {\em arXiv:1506.04139}, 2015.

\bibitem{Dasbach}
O.~T. Dasbach, D.~Futer, E.~Kalfagianni, X.-S. Lin, and N.~W. Stoltzfus.
\newblock Alternating sum formulae for the determinant and other link
  invariants.
\newblock {\em J. Knot Theory Ramifications}, 19(6):765--782, 2010.

\bibitem{Guts}
D.~Futer, E.~Kalfagianni, and J.~Purcell.
\newblock {\em Guts of surfaces and the colored {J}ones polynomial}, volume
  2069 of {\em Lecture Notes in Mathematics}.
\newblock Springer, Heidelberg, 2013.

\bibitem{FKPtwists}
D.~Futer, E.~Kalfagianni, and J.~S. Purcell.
\newblock Dehn filling, volume, and the {J}ones polynomial.
\newblock {\em J. Differential Geom.}, 78(3):429--464, 2008.

\bibitem{FKP}
D.~Futer, E.~Kalfagianni, and J.~S. Purcell.
\newblock Symmetric links and {C}onway sums: volume and {J}ones polynomial.
\newblock {\em Math. Res. Lett.}, 16(2):233--253, 2009.

\bibitem{BanachSurvey}
D.~Futer, E.~Kalfagianni, and J.~S. Purcell.
\newblock Jones polynomials, volume and essential knot surfaces: a survey.
\newblock In {\em Knots in {P}oland. {III}. {P}art 1}, volume 100 of {\em
  Banach Center Publ.}, pages 51--77. Polish Acad. Sci. Inst. Math., Warsaw,
  2014.

\bibitem{Kauffman}
L.~H. Kauffman.
\newblock {\em On knots}, volume 115 of {\em Annals of Mathematics Studies}.
\newblock Princeton University Press, Princeton, NJ, 1987.

\bibitem{LackenbyBound}
M.~Lackenby.
\newblock The volume of hyperbolic alternating link complements.
\newblock {\em Proc. London Math. Soc. (3)}, 88(1):204--224, 2004.
\newblock With an appendix by Ian Agol and Dylan Thurston.

\bibitem{Lickorish}
W.~B.~R. Lickorish.
\newblock {\em An introduction to knot theory}, volume 175 of {\em Graduate
  Texts in Mathematics}.
\newblock Springer-Verlag, New York, 1997.

\bibitem{Thistlethwaite}
W.~B.~R. Lickorish and M.~B. Thistlethwaite.
\newblock Some links with nontrivial polynomials and their crossing-numbers.
\newblock {\em Comment. Math. Helv.}, 63(4):527--539, 1988.

\bibitem{DThurston}
D.~Thurston.
\newblock Hyperbolic volume and the jones polynomial.
\newblock http://pages.iu.edu/~dpthurst/speaking/Grenoble.pdf, 2000.
\newblock Accessed: 2015-06-11.

\bibitem{ThurstonNotes}
W.~P. Thurston.
\newblock Three-dimensional manifolds, {K}leinian groups and hyperbolic
  geometry.
\newblock {\em Bull. Amer. Math. Soc. (N.S.)}, 6(3):357--381, 1982.

\end{thebibliography}

\end{document}